\documentclass{mcom-l}
\usepackage{amsmath}
\usepackage{amssymb}
\usepackage{amsrefs}
\usepackage{hyperref}

\newcommand{\arxivlink}[1]{\href{https://arxiv.org/abs/#1}{arXiv:#1}}

\BibSpec{arxiv}{
	+{}  {\PrintAuthors}                {author}
	+{,} { \textit}                     {title}
	+{.} { }                            {part}
	+{:} { \textit}                     {subtitle}
	+{.} { }                            {contribution}
	+{.} { }                            {partial}
	+{}  {\PrintDatePV}                 {date}
	+{,} {\arxivlink}                   {eprint}
	+{,} {\PrintDOI}                    {doi}
	+{.} { }                            {note}
}

\newtheorem{theorem}{Theorem}[section]
\newtheorem{lemma}[theorem]{Lemma}
\newtheorem{corollary}[theorem]{Corollary}
\theoremstyle{definition}

\newtheorem{proposition}[theorem]{Proposition}
\theoremstyle{remark}
\newtheorem{remark}[theorem]{Remark}

\numberwithin{equation}{section}
\usepackage{graphicx}
\usepackage[ruled,vlined]{algorithm2e}
\usepackage{subcaption}
\newcommand{\inner}[1]{\left\langle#1\right\rangle}
\newcommand{\norm}[1]{\left\|#1\right\|}
\newcommand{\E}{\mathcal{E}}
\newcommand{\G}{\mathcal{G}}
\newcommand{\zero}{\mathrm{zero}}
\begin{document}

\title[Adaptive Symplectic Proximal Point Algorithm]{An Adaptive Symplectic Proximal Point Algorithm: Faster Convergence and Improved Numerical Performance}



\author{Yi Zhang}
\address{School of Mathematics, Yunnan Normal University, 650092 Kunming, Yunnan, China.
\newline	
\indent Yunnan Key Laboratory of Modern Analytical Mathematics and Applications, Kunming, 650500, People’s Republic of China.}
\email{zhangyi2020@lsec.cc.ac.cn}
\thanks{}



\date{}

\dedicatory{}

\begin{abstract}
	The proximal point algorithm (PPA) stands as a fundamental approach for solving monotone inclusion problems. Notably, several key convex optimization algorithms have been proven to be specific instances of PPA. Given the importance of the PPA, there has been growing interest in developing its accelerated variants. However, for some specific cases, the PPA converges faster than the accelerated PPAs.  In this paper, we mainly study an adaptive version of symplectic accelerated PPA, called adaptive symplectic proximal point algorithm (ASPPA). We first prove that the convergence rate of ASPPA with respect to the square norm term is $O(1/k^2)$, which is the same as the convergence rate of some accelerated PPAs. Also, we show that ASPPA exists exponential convergence rate when solving strongly monotone inclusion problem. Our numerical experiments indicate that the performance of ASPPA is very good.
\end{abstract}

\maketitle
\section{Introduction}

In this paper, we consider the following monotone inclusion problem:
\begin{equation}
	\text{find } x^*\in\mathcal{H} \text{ such that } 0\in A(x^*),
	\label{eq:main}
\end{equation}
where $\mathcal{H}$ is a real Hilbert space equipped with inner product $\inner{\cdot,\cdot}$ and the induced norm $\norm{\cdot}=\sqrt{\inner{\cdot,\cdot}}$, and $A:\mathcal{H}\to 2^{\mathcal{H}}$ is a set-valued operator. We assume that $A$ is maximally monotone. Recall that a set-valued operator $A$ is monotone if
\[
\inner{u-v,x-y}\geq0,\quad\forall x,y\in\mathcal{H},\ u\in A(x),\ v\in A(y).
\]
A monotone operator $A$ is said to be maximal if its graph, defined as $\mathrm{graph}(A):=\{(x,u)\mid u\in A(x)\}$, is not a proper subset of the graph of any other monotone operator. We also assume that $\mathrm{zero}(A):=\{x\in\mathcal{H}\mid 0\in A(x)\}$ is non-empty.

The motivation for studying \eqref{eq:main} comes from its strong connection to convex programming. A fundamental algorithm for solving \eqref{eq:main} is the proximal point algorithm (PPA). Many important convex optimization algorithms have been found to be closely related to PPA. For example, the augmented Lagrangian method (ALM), the alternating direction method of multipliers (ADMM), and the Chambolle--Pock method (also known as the primal-dual hybrid gradient method, PDHG) have been proven to be special instances of PPA; see \cite{rockafellar78, eckstein92, chambolle11, chambolle16, connor20}. In addition, several first-order methods for solving variational inequality problems, such as the extra-gradient method and the optimistic gradient descent ascent method, are also closely related to PPA, as detailed in \cite{mokhtari20}. For more theory and applications of PPA, we refer to the textbooks \cite{bauschke17,ryu20} and the overview \cite{cai2022}.

To give a concise overview of PPA and its extensions, we first introduce some basic concepts. The proximal point operator of a maximally monotone operator $A$ is defined as $J_{sA}:=(I+sA)^{-1}$, where $s>0$ is a parameter. Specifically, if $u=J_{sA}(x)$, then $x-u\in sA(u)$. By Proposition 23.8 in \cite{bauschke17} and Minty's surjective theorem \cite{minty62}, $J_{sA}$ is a single-valued mapping defined on $\mathcal{H}$. Moreover, by the definition of $J_{sA}$, $x^*$ is a solution of \eqref{eq:main} if and only if $x^*=J_{sA}(x^*)$. Throughout this paper, we assume that $J_{sA}$ is easy to compute for some specific $s$. For simplicity, we assume $s=1$.

The iteration formula of the proximal point algorithm (PPA) is given by:
\begin{equation}
	x_{k+1} = J_A(x_k),
	\label{eq:PPA}
\end{equation}
which can be viewed as applying the Banach-Picard iteration to \(J_A\). In \cite{baillon96}, it was shown that PPA exhibits an \(O(1/k)\) convergence rate for \(\|x_{k+1} - x_k\|^2\). The quantity $\norm{x_{k+1}-x_k}^2$ can characterize the convergence rate of PPA stems from the property that \(x_k - x_{k+1} \in A(x_{k+1})\). More recently, \cite{gu20} provided a tighter convergence rate bound:
\[
\|x_{k+1} - x_k\|^2 \leq \frac{1}{(k+1)\left(1+\frac{1}{k}\right)^k} \cdot \text{dist}(x_0, \zero(A))^2,
\]
where \(\text{dist}(x_0, \zero(A)) = \inf_{x \in \zero(A)} \|x_0 - x\|\). We refer to the textbooks \cite{bauschke17} and \cite{ryu20} for the proof of the convergence rate of PPA.

Nesterov's accelerated gradient method \cite{nesterov83}, which achieves a faster convergence rate than standard gradient descent while maintaining the same computational complexity, has stimulated extensive research into accelerated variants of PPA. We first introduce a well-known acceleration technique: Halpern iteration. Halpern iteration, introduced in \cite{halpern67}, is an effective method for finding fixed points of nonexpansive operators \(T\). Its iteration scheme is
\begin{equation}
	x_{k+1} = \alpha_k x_0 + (1 - \alpha_k) T(x_k), \quad \alpha_k \in (0,1).
	\label{eq:halperniteration}
\end{equation}
Studies such as \cite{lieder21} and \cite{qi21} showed, from different perspectives, that the convergence rate of Halpern iteration for \(\norm{x_k - T(x_k)}^2\) is \(O(1/k^2)\). Since $T=2J_A-I$ is nonexpansive (Corollary 23.11 in \cite{bauschke17}), we obtain the accelerated PPA originally proposed in \cite{kim21} by setting $T=2J_A-I$ in \eqref{eq:halperniteration}; see Section 12.2 of \cite{ryu20} for details. Several works use Halpern iteration to accelerate other algorithms, such as \cite{cai24, lee21, park22, tran21, yoon21}. Some works focus on using Halpern iteration to develop efficient algorithms for specific problems, such as \cite{zhang24,liang24,chen2025,lu2024,sun2025a}. Moreover, several works further extend Halpern iteration. For example, \cite{he2024} developed an adaptive version of Halpern iteration and showed that its convergence rate with respect to $\norm{x_k-T(x_k)}^2$ is at least $O(1/k^2)$. In addition, \cite{park22} proposed OC-Halpern for finding fixed points of contractions and OS-PPM for solving \eqref{eq:main} when $A$ is also strongly monotone, i.e., there exists $\mu>0$ such that \[
\inner{u-v,x-y}\geq\mu\norm{x-y}^2,\quad\forall x, y\in\mathcal{H}, u\in A(x), v\in A(y).
\]

Several other acceleration techniques for PPA have also been developed. Here, we introduce two approaches derived from ordinary differential equations (ODEs). Early work by \cite{su16} linked Nesterov's accelerated gradient method with a second-order ODE and introduced the Lyapunov function technique to demonstrate the $O(1/k^2)$ convergence rate of the function value for the Nesterov's accelerated gradient method. Recent studies have increasingly utilized Lyapunov functions to establish convergence properties of optimization algorithms; see, e.g., \cite{attouch16,chen23,li2026,shi19}. Given the theoretical depth of ODEs, several studies have focused on deriving accelerated versions of PPA from discretizations of ODEs. For instance, \cite{attouch192, attouch21} used the Yosida approximation of $A$ to propose second-order ODEs and then obtained accelerated PPAs by properly discretizing these ODEs. However, the parameter of the proximal point operator in the resulting accelerated PPAs changes with the iterations, which may limit their applications in deriving accelerated convex optimization algorithms; see \cite{eckstein92} for a discussion. More recently, \cite{bot23} proposed a new accelerated PPA framework called Fast-KM. They used the ODE proposed in \cite{bot232} to obtain a new ODE and applied the forward Euler method to obtain Fast-KM. They proved that the convergence rate of Fast-KM for $\norm{x_k-J_A(x_k)}^2$ is $o(1/k^2)$. In addition, \cite{yuan2026} studied a new ODE for \eqref{eq:main} based on the ODE in \cite{yuan2025} and used the symplectic Euler method to discretize it, obtaining the symplectic PPA (SPPA). The authors also proved that the convergence rate of SPPA for $\norm{a_k}^2$, where $a_k\in A(x_k)$, is $o(1/k^2)$.

\subsection{Our Motivations and Contributions}

Motivated by the adaptive Halpern iteration in \cite{he2024}, which improves the numerical performance of Halpern iteration, we develop an adaptive version of the symplectic proximal point algorithm (SPPA), called the adaptive symplectic proximal point algorithm (ASPPA, Algorithm \ref{al:sppa}). Unlike SPPA, ASPPA adaptively selects the parameter $\varphi_{k+1}$ by solving a quadratic equation, allowing the algorithm to adjust its acceleration according to the local geometry of the problem.

Our main contributions are summarized as follows.

\begin{itemize}
	\item We prove that the squared norm term $\norm{a_k}^2$ generated by ASPPA converges at rate $O(1/k^2)$, which matches the convergence rate of several accelerated PPAs. Moreover, under an additional summability condition, we further establish $\lim_{k\to\infty} k^2\norm{a_k}^2=0$, i.e., an $o(1/k^2)$ convergence rate.
	\item We establish that if $A$ is $\mu$-strongly monotone, ASPPA exhibits a linear (exponential) convergence rate. Compared with OS-PPM \cite{park22}, our algorithm does not require prior knowledge of the strong monotonicity parameter $\mu$.
	\item We conduct numerical experiments on strongly monotone inclusions, LASSO, and basis pursuit. The results show that ASPPA and its induced ADMM/PDHG variants often exhibit faster initial convergence and better numerical performance than several existing accelerated methods.
\end{itemize}
\section{Adaptive Symplectic Proximal Point Algorithm}
\begin{algorithm}[H]
	\caption{Adaptive Symplectic Proximal Point Algorithm, ASPPA}
	\label{al:sppa}
	\textbf{Initialize:} $x_0, z_0=x_0$, $\varphi_0=0$, $a_0=0$, $r>1$, $c\in (0,r-1)$, $\alpha\in[0,1)$.
	\For{$k=0, 1, \cdots$}{
		$\tilde{x}_{k+1}=\dfrac{\varphi_k}{\varphi_k+r}x_k+\dfrac{r}{\varphi_k+r}z_k$.\\
		$x_{k+1}=J_A(\tilde{x}_{k+1})$.\\
		$a_{k+1}=\tilde{x}_{k+1}-x_{k+1}$.\\
		\If{$a_{k+1}=0$}{
		Break the loop.\\
		}
		Solving $\varphi_{k+1}$ such that\[
		f_{k+1}(\varphi_{k+1}):=g_{k+1}^1\varphi^2_{k+1}+g_{k+1}^2\varphi_{k+1}+g_{k+1}^3=0\text{ with }  \varphi_{k+1}\geq\varphi_k+1,
		\]
		where
		\begin{align*}
			g_{k+1}^1 =&\  \frac{(r+1)c+1}{2r^2}\norm{a_{k+1}}^2,  \\
			g_{k+1}^2 =&\ \frac{r(1-2c\varphi_k)+2\alpha(r-1-c)\varphi_k}{2r^2}\norm{a_{k+1}}^2+\inner{a_{k+1},x_{k+1}-z_k},\\
			g_{k+1}^3 =&\ \frac{(r-1)\varphi_k^2c-(\varphi_k+2r)(\varphi_k+r)}{2r^2}\norm{a_{k+1}}^2-(\varphi_k+1)\inner{a_{k+1},x_{k+1}-z_k}\\
			&\ -\frac{\varphi_k(\varphi_k+r)}{r^2}\inner{a_{k+1}-a_k,x_{k+1}-x_k}-\frac{\varphi_k(\varphi_k+r)}{2r^2}\norm{a_{k+1}-a_k}^2\\
			&\ -\frac{\alpha[2(r-1-c)\varphi_k^2-2r^2+(r+1)(c+1)]}{2r^2}\norm{a_{k+1}}^2.
		\end{align*}
		$z_{k+1}=z_k-\dfrac{c(\varphi_{k+1}-\varphi_k)}{r}a_{k+1}$.
	}
	\textbf{Output:} $x_{k+1}$.
\end{algorithm}

\section{Convergence Results of Algorithm \ref{al:sppa}}
\subsection{Well-definedness of Algorithm \ref{al:sppa}}

Before studying the convergence rate of Algorithm \ref{al:sppa}, we need to show that the algorithm is well-defined, i.e., if $x_{k+1}\notin\zero(A)$, there exists $\varphi_{k+1}\geq\varphi_k+1$ such that $f_{k+1}(\varphi_{k+1})=0$.

\begin{lemma}
	\label{lem:proper}
	Algorithm \ref{al:sppa} is well-defined, i.e., if $x_{k+1}\notin\zero(A)$, there exists $\varphi_{k+1}\geq\varphi_k+1$ such that $f_{k+1}(\varphi_{k+1})=0$.
\end{lemma}

\begin{proof}
	First, according to the definition of $g_{k+1}^1$, $g_{k+1}^2$, $g_{k+1}^3$, the function $f_{k+1}$ can be written as
	\begin{align*}
		&\ f_{k+1}(\varphi_{k+1})\\
		=&\ \left(\frac{\varphi_{k+1}(\varphi_{k+1}+r)}{2r^2}+\frac{\varphi_{k+1}(\varphi_{k+1}-\varphi_k)c}{r^2}+\frac{(r-1)(\varphi_{k+1}-\varphi_k)^2c}{2r^2}-\frac{(\varphi_k+2r)(\varphi_k+r)}{2r^2}\right)\norm{a_{k+1}}^2 \\
		&\ +(\varphi_{k+1}-\varphi_k-1)\inner{a_{k+1},x_{k+1}-z_k}-\frac{\varphi_k(\varphi_k+r)}{r^2}\inner{a_{k+1}-a_k,x_{k+1}-x_k}\\
		&\ -\frac{\varphi_k(\varphi_k+r)}{2r^2}\norm{a_{k+1}-a_k}^2+\frac{\alpha[2(r-1-c)(\varphi_{k+1}-\varphi_k)\varphi_k+2r^2-(r+1)(c+1)]}{2r^2}\norm{a_{k+1}}^2.
	\end{align*}
	Taking $\varphi_{k+1}=\varphi_k+1$, we obtain
	\begin{align*}
		f_{k+1}(\varphi_k+1)
		=&\ \left(\frac{(\varphi_k+1)(\varphi_k+r+1)}{2r^2}+\frac{(\varphi_k+1)c}{r^2}+\frac{(r-1)c}{2r^2}-\frac{(\varphi_k+2r)(\varphi_k+r)}{2r^2}\right)\norm{a_{k+1}}^2\\
		&\ -\frac{\varphi_k(\varphi_k+r)}{2r^2}\norm{a_{k+1}-a_k}^2-\frac{\varphi_k(\varphi_k+r)}{r^2}\inner{a_{k+1}-a_k,x_{k+1}-x_k}\\
		&\ +\frac{\alpha[2(r-1-c)\varphi_k+2r^2-(r+1)(c+1)]}{2r^2}\norm{a_{k+1}}^2\\
		=&\ -\frac{(1-\alpha)[2(r-1-c)\varphi_k+2r^2-(r+1)(c+1)]}{2r^2}\norm{a_{k+1}}^2\\
		&\ -\frac{\varphi_k(\varphi_k+r)}{2r^2}\norm{a_{k+1}-a_k}^2-\frac{\varphi_k(\varphi_k+r)}{r^2}\inner{a_{k+1}-a_k,x_{k+1}-x_k}.
	\end{align*}
	Since $c<r-1$, we have $r-1-c>0$ and
	\[
	2r^2-(r+1)(c+1)>2r^2-(r+1)r=r^2-r>0.
	\]
	As $\alpha\in[0,1)$, we have $1-\alpha>0$, and hence
	\[
	-\frac{(1-\alpha)[2(r-1-c)\varphi_k+2r^2-(r+1)(c+1)]}{2r^2}<0.
	\]
	If $k\geq 1$, then from $x_{k+1}=J_A(\tilde{x}_{k+1})$ we have $a_{k+1}=\tilde{x}_{k+1}-x_{k+1}\in A(x_{k+1})$, and similarly $a_k\in A(x_k)$. Since $A$ is monotone,
	\[
	\inner{a_{k+1}-a_k,x_{k+1}-x_k}\geq 0.
	\]
	Therefore the last two terms in $f_{k+1}(\varphi_k+1)$ are non-positive, and the first term is strictly negative. Since $x_{k+1}\notin\zero(A)$, we have $a_{k+1}\neq 0$, which implies
	\[
	f_{k+1}(\varphi_k+1)<0.
	\]
	If $k=0$, then $\varphi_0=0$ and $a_0=0$, so
	\[
	f_{1}(1)=-\frac{(1-\alpha)[2r^2-(r+1)(c+1)]}{2r^2}\norm{a_1}^2<0.
	\]
	In both cases, since $g_{k+1}^1=\frac{(r+1)c+1}{2r^2}\norm{a_{k+1}}^2>0$, the function $f_{k+1}$ is a quadratic function opening upwards. Together with $f_{k+1}(\varphi_k+1)<0$ and $f_{k+1}(\varphi)\to+\infty$ as $\varphi\to+\infty$, there exists $\varphi_{k+1}\geq\varphi_k+1$ such that $f_{k+1}(\varphi_{k+1})=0$. This completes the proof.
\end{proof}

As a result of Lemma \ref{lem:proper}, the recursive rule of $\varphi_{k+1}$ is
\begin{equation}
	\varphi_{k+1}=\frac{-g_{k+1}^2+\sqrt{(g_{k+1}^2)^2-4g_{k+1}^1g_{k+1}^3}}{2g_{k+1}^1}.
\end{equation}
Here the larger root is chosen so that $\varphi_{k+1}\geq\varphi_k+1$, and the existence of such a root is guaranteed by Lemma \ref{lem:proper}.
\subsection{Last-Iterate Convergence Rate}

Next, we study the convergence rate of Algorithm \ref{al:sppa}. Similar to the proof of the convergence rate of SPPA, we consider the following Lyapunov function:
\begin{equation}
	\E(k):=\frac{(\varphi_k+r)\varphi_k}{2r^2}\norm{a_k}^2+\frac{\varphi_k}{r}\inner{a_k,x_k-z_k}+\frac{r-1}{2c}\norm{z_k-x^*}^2,
	\label{eq:lyapunovdiscrete}
\end{equation}
where $x^*\in\zero(A)$.

\begin{lemma}
	\label{lem:lyapunovdiscretenonnegative}
	Let $A$ be a maximally monotone operator, and let $\{x_k\}$ and $\{z_k\}$ be the sequences generated by Algorithm \ref{al:sppa}. The Lyapunov function $\E(k)$ given in \eqref{eq:lyapunovdiscrete} can be represented as
	\begin{align*}
		\E(k)=&\ \frac{(r-1-c)\varphi_k^2+r(r-1)\varphi_k}{2r^2(r-1)}\norm{a_k}^2+\frac{\varphi_k}{r}\inner{a_k,x_k-x^*}\\
		&\ +\frac{1}{2}\norm{\frac{\varphi_k}{r}\sqrt{\frac{c}{r-1}}a_k+\sqrt{\frac{r-1}{c}}(x^*-z_k)}^2.
	\end{align*}
	If $r>1$ and $0<c<r-1$, then $\E(k)$ is nonnegative.
\end{lemma}

\begin{proof}
	Using the equality
	\[
	\inner{x,y}=\frac{1}{2}\norm{x+y}^2-\frac{1}{2}\norm{x}^2-\frac{1}{2}\norm{y}^2
	\]
	with
	\begin{align*}
		x &=  \sqrt{\frac{c\varphi_k^2}{r^2(r-1)}}a_k, \\
		y &= \sqrt{\frac{r-1}{c}}(x^*-z_k),
	\end{align*}
	we obtain
	\begin{align*}
		\E(k)= &\ \frac{\varphi_k(\varphi_k+r)}{2r^2}\norm{a_k}^2+\frac{\varphi_k}{r}\inner{a_k,x_k-x^*}+\frac{\varphi_k}{r}\inner{a_k,x^*-z_k}+\frac{r-1}{2c}\norm{z_k-x^*}^2\\
		=&\ \frac{(r-1-c)\varphi_k^2+r(r-1)\varphi_k}{2r^2(r-1)}\norm{a_k}^2+\frac{\varphi_k}{r}\inner{a_k,x_k-x^*}\\
		&\ +\frac{1}{2}\norm{\sqrt{\frac{c\varphi_k^2}{r^2(r-1)}}a_k+\sqrt{\frac{r-1}{c}}(x^*-z_k)}^2.
	\end{align*}
	Since $0<c<r-1$,
	\[
	\frac{(r-1-c)\varphi_k^2+r(r-1)\varphi_k}{2r^2(r-1)}\geq 0,\quad\forall k\geq 0.
	\]
	Moreover, because $x^*\in\zero(A)$ and $A$ is monotone, we have
	\[
	\inner{a_k,x_k-x^*}=\inner{a_k-0,x_k-x^*}\geq 0.
	\]
	Therefore $\E(k)$ is a sum of three nonnegative terms, which implies $\E(k)\geq 0$.
\end{proof}

\begin{theorem}
	\label{thm:lyapunovdiscrete}
	Let $A$ be a maximally monotone operator, and let $\{x_k\}$ and $\{z_k\}$ be the sequences generated by Algorithm \ref{al:sppa}. If $r>1$, $0<c<r-1$, and $\alpha\in[0,1)$, then
	\begin{equation}
		\label{eq:differenceestimation}
		\begin{split}
			&\ \E(k+1)-\E(k)\\
			=&\ -\frac{\alpha[2(r-1-c)(\varphi_{k+1}-\varphi_k)\varphi_k+2r^2-(r+1)(c+1)]}{2r^2}\norm{a_{k+1}}^2\\
			&\ -\frac{(r-1)(\varphi_{k+1}-\varphi_k)}{r}\inner{a_{k+1},x_{k+1}-x^*}\\
			\leq&\ 0.
		\end{split}
	\end{equation}
	Also,
	\begin{equation}
		\norm{a_k}^2\leq\frac{r^2(r-1)^2\norm{x_0-x^*}^2}{c[(r-1-c)\varphi_k^2+r(r-1)\varphi_k]}.
		\label{eq:lastiterate}
	\end{equation}
	Moreover, if $\alpha>0$, then
	\begin{equation}
		\sum_{k=0}^\infty(\varphi_{k+1}-\varphi_k)\varphi_k\norm{a_{k+1}}^2<+\infty.
		\label{eq:sumalpha}
	\end{equation}
\end{theorem}

\begin{proof}
	Step 1: Decompose the difference $\E(k+1)-\E(k)$ into three parts:
	\begin{align*}
		\E(k+1)-\E(k)= &\  \frac{\varphi_{k+1}(\varphi_{k+1}+r)}{2r^2}\norm{a_{k+1}}^2-\frac{\varphi_k(\varphi_k+r)}{2r^2}\norm{a_k}^2\\
		&\ +\underbrace{\frac{\varphi_{k+1}}{r}\inner{a_{k+1},x_{k+1}-z_{k+1}}-\frac{\varphi_k}{r}\inner{a_k,x_k-z_k}}_{\text{I}}\\
		&\ +\underbrace{\frac{r-1}{2c}\left(\norm{z_{k+1}-x^*}^2-\norm{z_k-x^*}^2\right)}_{\text{II}}.
	\end{align*}
	
	Step 2: Transform I and II. First, for I, using $z_{k+1}=z_k-\frac{c(\varphi_{k+1}-\varphi_k)}{r}a_{k+1}$, we have
	\begin{align*}
		\text{I}= &\ \frac{\varphi_{k+1}}{r}\inner{a_{k+1},x_{k+1}-z_k}-\frac{\varphi_k}{r}\inner{a_k,x_k-z_k}+\frac{\varphi_{k+1}(\varphi_{k+1}-\varphi_k)c}{r^2}\norm{a_{k+1}}^2\\
		=&\ \frac{\varphi_{k+1}-\varphi_k}{r}\inner{a_{k+1},x_{k+1}-z_k}+\frac{\varphi_k}{r}\inner{a_{k+1},x_{k+1}-x_k}+\frac{\varphi_k}{r}\inner{a_{k+1}-a_k,x_k-z_k}\\
		&\ +\frac{\varphi_{k+1}(\varphi_{k+1}-\varphi_k)c}{r^2}\norm{a_{k+1}}^2.
	\end{align*}
	By using the equality
	\[
	\varphi_k(x_{k+1}-x_k)=r(z_k-x_{k+1})-(\varphi_k+r)a_{k+1},
	\]
	we have
	\[
	\frac{\varphi_k}{r}\inner{a_{k+1}, x_{k+1}-x_k}=\frac{1}{r}\inner{a_{k+1},-(\varphi_k+r)a_{k+1}+r(z_k-x_{k+1})}.
	\]
	Also, from the equality
	\[
	r(z_k-x_k)=(\varphi_k+r)(a_{k+1}+x_{k+1}-x_k),
	\]
	it follows that
		\[
	\frac{\varphi_k}{r}\inner{a_{k+1}-a_k,x_k-z_k}=-\frac{\varphi_k(\varphi_k+r)}{r^2}\inner{a_{k+1}-a_k,a_{k+1}+x_{k+1}-x_k}.
	\]
	Based on the previous discussion, we obtain
	\begin{align*}
		\text{I}= &\  \frac{\varphi_{k+1}-\varphi_k-r}{r}\inner{a_{k+1},x_{k+1}-z_k}-\frac{\varphi_k+r}{r}\norm{a_{k+1}}^2-\frac{\varphi_k(\varphi_k+r)}{r^2}\inner{a_{k+1}-a_k,a_{k+1}}\\
		&\ -\frac{\varphi_k(\varphi_k+r)}{r^2}\inner{a_{k+1}-a_k,x_{k+1}-x_k}+\frac{\varphi_{k+1}(\varphi_{k+1}-\varphi_k)c}{r^2}\norm{a_{k+1}}^2\\
		=&\ (\varphi_{k+1}-\varphi_k-1)\inner{a_{k+1},x_{k+1}-z_k}-\frac{(r-1)(\varphi_{k+1}-\varphi_k)}{r}\inner{a_{k+1},x_{k+1}-z_k}\\
		&\ -\left(\frac{\varphi_k(\varphi_k+r)}{2r^2}+\frac{\varphi_k+r}{r}-\frac{\varphi_{k+1}(\varphi_{k+1}-\varphi_k)c}{r^2}\right)\norm{a_{k+1}}^2+\frac{\varphi_k(\varphi_k+r)}{2r^2}\norm{a_k}^2\\
		&\ -\frac{\varphi_k(\varphi_k+r)}{2r^2}\norm{a_{k+1}-a_k}^2-\frac{\varphi_k(\varphi_k+r)}{r^2}\inner{a_{k+1}-a_k,x_{k+1}-x_k}.
	\end{align*}
	Now for II:
	\begin{align*}
		\text{II} & =\frac{r-1}{c}\inner{z_{k+1}-z_k,z_k-x^*}+\frac{r-1}{2c}\norm{z_{k+1}-z_k}^2\\
		&=-\frac{(r-1)(\varphi_{k+1}-\varphi_k)}{r}\inner{a_{k+1},z_k-x^*}+\frac{(r-1)(\varphi_{k+1}-\varphi_k)^2c}{2r^2}\norm{a_{k+1}}^2.
	\end{align*}
	
	Step 3: Estimate the upper bound of $\E(k+1)-\E(k)$. Combining the above calculations yields
	\begin{align*}
		&\ \E(k+1)-\E(k)\\
		=&\ \left(\frac{\varphi_{k+1}(\varphi_{k+1}+r)}{2r^2}+\frac{\varphi_{k+1}(\varphi_{k+1}-\varphi_k)c}{r^2}+\frac{(r-1)(\varphi_{k+1}-\varphi_k)^2c}{2r^2}-\frac{(\varphi_k+2r)(\varphi_k+r)}{2r^2}\right)\norm{a_{k+1}}^2 \\
		&\ +(\varphi_{k+1}-\varphi_k-1)\inner{a_{k+1},x_{k+1}-z_k}-\frac{(r-1)(\varphi_{k+1}-\varphi_k)}{r}\inner{a_{k+1},x_{k+1}-x^*}\\
		&\ -\frac{\varphi_k(\varphi_k+r)}{r^2}\inner{a_{k+1}-a_k,x_{k+1}-x_k}-\frac{\varphi_k(\varphi_k+r)}{2r^2}\norm{a_{k+1}-a_k}^2.
	\end{align*}
	Using the definition of $\varphi_{k+1}$, we obtain
	\begin{align*}
		&\ \E(k+1)-\E(k)\\
		=&\ -\frac{\alpha[2(r-1-c)(\varphi_{k+1}-\varphi_k)\varphi_k+2r^2-(r+1)(c+1)]}{2r^2}\norm{a_{k+1}}^2\\
		&\ -\frac{(r-1)(\varphi_{k+1}-\varphi_k)}{r}\inner{a_{k+1},x_{k+1}-x^*}\\
		\leq&\ 0.
	\end{align*}
	
	Step 4: Derive the convergence rate. Since $\{\E(k)\}$ is non-increasing and Lemma \ref{lem:lyapunovdiscretenonnegative} holds, we have
	\[
	\frac{(r-1-c)\varphi_k^2+r(r-1)\varphi_k}{2r^2(r-1)}\norm{a_k}^2\leq\E(k)\leq\E(0)=\frac{r-1}{2c}\norm{x_0-x^*}^2.
	\]
	Dividing both sides by $\frac{(r-1-c)\varphi_k^2+r(r-1)\varphi_k}{2r^2(r-1)}$, we obtain \eqref{eq:lastiterate}. Also, by summing the inequality in \eqref{eq:differenceestimation}, we have
	\[
	\sum_{k=0}^\infty(\varphi_{k+1}-\varphi_k)\varphi_k\norm{a_{k+1}}^2\leq\frac{r^2\E(0)}{\alpha(r-1-c)}<+\infty.
	\]
\end{proof}

\begin{corollary}
	\label{coro:1}
	Let $A$ be a maximally monotone operator, and let $\{x_k\}$ and $\{z_k\}$ be the sequences generated by Algorithm \ref{al:sppa}. If $r>1$, $0<c<r-1$, and $\alpha\in(0,1)$, then
	\begin{equation}
		\liminf_{k\to\infty}k^2\norm{a_k}^2=0.
	\end{equation}
\end{corollary}

\begin{proof}
	From \eqref{eq:sumalpha}, we have
	\[
	\sum_{k=0}^\infty(\varphi_{k+1}-\varphi_k)\varphi_k\norm{a_{k+1}}^2<+\infty.
	\]
	Since $\varphi_{k+1}\geq\varphi_k+1$ and $\varphi_0=0$, we have $\varphi_k\geq k$. Hence
	\[
	(\varphi_{k+1}-\varphi_k)\varphi_k\geq\varphi_k\geq k,
	\]
	which implies
	\[
	\sum_{k=0}^\infty k\norm{a_{k+1}}^2
	\leq\sum_{k=0}^\infty(\varphi_{k+1}-\varphi_k)\varphi_k\norm{a_{k+1}}^2
	<+\infty.
	\]
	Suppose, for contradiction, that $\liminf_{k\to\infty}k^2\norm{a_k}^2>0$. Then there exist $\varepsilon>0$ and $K\geq 1$ such that
	\[
	k^2\norm{a_k}^2\geq\varepsilon,\quad\forall k\geq K.
	\]
	Thus
	\[
	k\norm{a_k}^2\geq\frac{\varepsilon}{k},\quad\forall k\geq K,
	\]
	and therefore
	\[
	\sum_{k=K}^\infty k\norm{a_k}^2
	\geq\varepsilon\sum_{k=K}^\infty\frac{1}{k}
	=+\infty,
	\]
	which contradicts $\sum_{k=0}^\infty k\norm{a_{k+1}}^2<+\infty$. Hence
	\[
	\liminf_{k\to\infty}k^2\norm{a_k}^2=0.
	\]
\end{proof}
\begin{remark}
	The parameter $\alpha$ plays a crucial role in ensuring the conclusion of Corollary \ref{coro:1}. Indeed, if $\alpha=0$, the result
	\[
	\liminf_{k\to\infty}k^2\norm{a_k}^2=0
	\]
	may fail to hold, which can lead to slower convergence in practice (see Section \ref{sec:3}). Although this $\liminf$ result does not by itself imply the $o(1/k^2)$ convergence rate of $\norm{a_k}^2$, it is beneficial for the numerical performance of Algorithm \ref{al:sppa}. In Appendix \ref{sec:a}, we provide a sufficient condition under which the stronger conclusion
	\[
	\lim_{k\to\infty}k^2\norm{a_k}^2=0
	\]
	holds.
\end{remark}
\begin{theorem}
	Let $\{x_k\}$ and $\{a_k\}$ be the sequences generated by Algorithm \ref{al:sppa}. If $A$ is $\mu$-strongly monotone with $\mu>0$, $r>1$, $c\in(0,r-1)$, $\alpha\in[0,1)$, then there exists $\sigma>0$ such that\[
	\varphi_{k+1}\geq(1+\sigma)\varphi_k+1.
	\]
	Moreover, if $\alpha\in(0,1)$,  the convergence rate of Algorithm \ref{al:sppa} is\[
	\norm{a_{k+1}}^2\leq o\left(\frac{1}{\varphi_k^2}\right).
	\]
\end{theorem}
\begin{proof}
	\begin{align*}
		&\ f_{k+1}(\varphi_{k+1})\\
		=&\  f_{k+1}(\varphi_{k+1})-\frac{(\varphi_{k+1}-\varphi_k-1)\varphi_k^2}{2r(\varphi_k+r)}\norm{x_{k+1}-x_k}^2+\frac{(\varphi_{k+1}-\varphi_k-1)\varphi_k^2}{2r(\varphi_k+r)}\norm{x_{k+1}-x_k}^2\\
		=&\ f_{k+1}(\varphi_{k+1})-\frac{\varphi_{k+1}-\varphi_k-1}{2r(\varphi_k+r)}\norm{r(z_k-x_{k+1})-(\varphi_k+r)a_{k+1}}^2\\
		&\ +\frac{(\varphi_{k+1}-\varphi_k-1)\varphi_k^2}{2r(\varphi_k+r)}\norm{x_{k+1}-x_k}^2\\
		=&\ \left(\frac{\varphi_{k+1}(\varphi_{k+1}+r)}{2r^2}+\frac{\varphi_{k+1}(\varphi_{k+1}-\varphi_k)c}{r^2}+\frac{(r-1)(\varphi_{k+1}-\varphi_k)^2c}{2r^2}-\frac{(\varphi_k+2r)(\varphi_k+r)}{2r^2}\right)\norm{a_{k+1}}^2 \\
		&\ -\frac{r(\varphi_{k+1}-\varphi_k-1)}{2(\varphi_k+r)}\norm{z_k-x_{k+1}}^2-\frac{(\varphi_{k+1}-\varphi_k-1)(\varphi_k+r)}{2r}\norm{a_{k+1}}^2\\
		&\ -\frac{\varphi_k(\varphi_k+r)}{r^2}\inner{a_{k+1}-a_k,x_{k+1}-x_k}-\frac{\varphi_k(\varphi_k+r)}{2r^2}\norm{a_{k+1}-a_k}^2\\
		&\ +\frac{(\varphi_{k+1}-\varphi_k-1)\varphi_k^2}{2r(\varphi_k+r)}\norm{x_{k+1}-x_k}^2\\
		&\ +\frac{\alpha[2(r-1-c)(\varphi_{k+1}-\varphi_k)\varphi_k+2r^2-(r+1)(c+1)]}{2r^2}\norm{a_{k+1}}^2\\
		\leq&\ \left(\frac{\varphi_{k+1}(\varphi_{k+1}+r)}{2r^2}+\frac{\varphi_{k+1}(\varphi_{k+1}-\varphi_k)c}{r^2}+\frac{(r-1)(\varphi_{k+1}-\varphi_k)^2c}{2r^2}-\frac{(\varphi_k+2r)(\varphi_k+r)}{2r^2}\right)\norm{a_{k+1}}^2 \\
		&\ -\frac{r(\varphi_{k+1}-\varphi_k-1)}{2(\varphi_k+r)}\norm{z_k-x_{k+1}}^2-\frac{(\varphi_{k+1}-\varphi_k-1)(\varphi_k+r)}{2r}\norm{a_{k+1}}^2\\
		&\ -\left(\frac{\varphi_k(\varphi_k+r)\mu}{r^2}+\frac{\varphi_k(\varphi_k+r)\mu^2}{2r^2}-\frac{(\varphi_{k+1}-\varphi_k-1)\varphi_k^2}{2r(\varphi_k+r)}\right)\norm{x_{k+1}-x_k}^2\\
		&\ +\frac{\alpha[2(r-1-c)(\varphi_{k+1}-\varphi_k)\varphi_k+2r^2-(r+1)(c+1)]}{2r^2}\norm{a_{k+1}}^2.
	\end{align*}
	
	Let $\varphi_{k+1}=(1+\sigma)\varphi_k+1$, where $\sigma>0$. We have
	\begin{align*}
		&\ f_{k+1}(\varphi_{k+1})\\
		=&\ -\left(h_1(\sigma)\varphi_k^2+h_2(\sigma)\varphi_k+\frac{(1-\alpha)[2r^2-(r+1)(c+1)]}{2r^2}\right)\norm{a_{k+1}}^2-\frac{\sigma r\varphi_k}{2(\varphi_k+r)}\norm{z_k-x_{k+1}}^2\\
		&\ -\left(\frac{\varphi_k(\varphi_k+r)\mu}{r^2}+\frac{\varphi_k(\varphi_k+r)\mu^2}{2r^2}-\frac{\sigma\varphi_k^3}{2r(\varphi_k+r)}\right)\norm{x_{k+1}-x_k}^2,
	\end{align*}
	where\begin{align*}
		h_1(\sigma) &= \frac{2(1-\alpha)(r-1-c)\sigma-[1+(r+1)c]\sigma^2}{2r^2}, \\
		h_2(\sigma) &= \frac{2(1-\alpha)(r-1-c)+[r^2-r-2-2(r+1)c]\sigma}{2r^2}.
	\end{align*}
	First, because of the inequality $\frac{\varphi_k^3}{\varphi_k+r}<\varphi_k(\varphi_k+r)$, then for all $\sigma\in\left(0,\frac{2\mu+\mu^2}{r}\right)$, we have\[
	-\left(\frac{\varphi_k(\varphi_k+r)\mu}{r^2}+\frac{\varphi_k(\varphi_k+r)\mu^2}{2r^2}-\frac{\sigma\varphi_k^3}{2r(\varphi_k+r)}\right)\norm{x_{k+1}-x_k}^2\leq 0.
	\]
	Second, if $\sigma\in\left(0,\frac{2(1-\alpha)(r-1-c)}{1+(r+1)c}\right]$, $h_1(\sigma)\geq 0$. Finally, if $r^2-r-2(r+1)c<0$, then for all $\sigma\in\left(0,-\frac{2(1-\alpha)(r-1-c)}{r^2-r-2-2(r+1)c}\right]$, $h_2(\sigma)\geq 0$. If $r^2-r-2(r+1)c\geq 0$, $h_2(\sigma)\geq 0$ holds for all $\sigma>0$. 
	
	In conclusion, if $r^2-r-2(r+1)c<0$, then for all \[
	\sigma\leq\min\left\{\frac{2\mu+\mu^2}{r},\frac{2(1-\alpha)(r-1-c)}{1+(r+1)c},-\frac{2(1-\alpha)(r-1-c)}{r^2-r-2-2(r+1)c}\right\},
	\]
	we have $f_{k+1}(\varphi_{k+1})\leq 0$. If $r^2-r-2(r+1)c\geq 0$, then for all \[
	\sigma\leq\min\left\{\frac{2\mu+\mu^2}{r},\frac{2(1-\alpha)(r-1-c)}{1+(r+1)c}\right\},
	\]
	we have $f_{k+1}(\varphi_{k+1})\leq 0$.
	
	If $\alpha\in(0,1)$, then according to \eqref{eq:sumalpha} and $\varphi_{k+1}\geq(1+\sigma)\varphi_k+1$, we have \[
	\sum_{k=1}^\infty\varphi_k^2\norm{a_{k+1}}^2<+\infty.
	\]
	Thus\[
	\lim_{k\to\infty}\varphi_k^2\norm{a_{k+1}}^2=0.
	\]
\end{proof}
\section{Numerical Experiments}
\label{sec:3}
\subsection{Strongly Monotone Case}

In this section, we consider the following maximally monotone operator:
\[
A(x)=\mu(MM^\top+\mathrm{Id})x,
\]
where $\mu>0$, $M\in\mathbb{R}^{2n\times n}$, and $\mathrm{Id}$ is the $2n\times 2n$ identity matrix. It is easy to see that $A$ is maximally monotone and $\mu$-strongly monotone, and that $\zero(A)=\{0\}$. Since $A$ is $\mu$-strongly monotone, the operator
\begin{equation}
	T:=\left(1+\frac{1}{1+2\mu}\right)J_A-\frac{1}{1+2\mu}\mathrm{Id},
	\label{eq:contraction}
\end{equation}
where $I$ is the identity operator, is a contraction with modulus $\frac{1}{1+2\mu}$, and the set of fixed points of $T$ coincides with $\zero(A)$ (see Lemma 2.1 in \cite{park22}).

In our experiment, we set $n=500$, $\mu=0.01$, and generate $M$ randomly. Six algorithms are tested: PPA, SPPA (with $r=20$, $c=(r-1)/2$), ASPPA (Algorithm \ref{al:sppa}, with $r=20$, $c=(r-1)/2$, $\alpha=0$), the Banach--Picard iteration with \eqref{eq:contraction} (labeled as BP), OS-PPM in \cite{park22}, and the adaptive Halpern iteration with \eqref{eq:contraction} (labeled as AH). The Fast-KM iteration is not considered in this section because its numerical performance is similar to that of SPPA (see \cite{yuan2026}). All algorithms start from the same randomly generated initial point. We run these algorithms for 10000 iterations and plot $\norm{x_k}$ and $\norm{a_k}$ against the iteration count in Figure \ref{fig:1}.

\begin{figure}[htbp]
	\centering
	\begin{subfigure}{0.45\linewidth}
		\includegraphics[width=\linewidth]{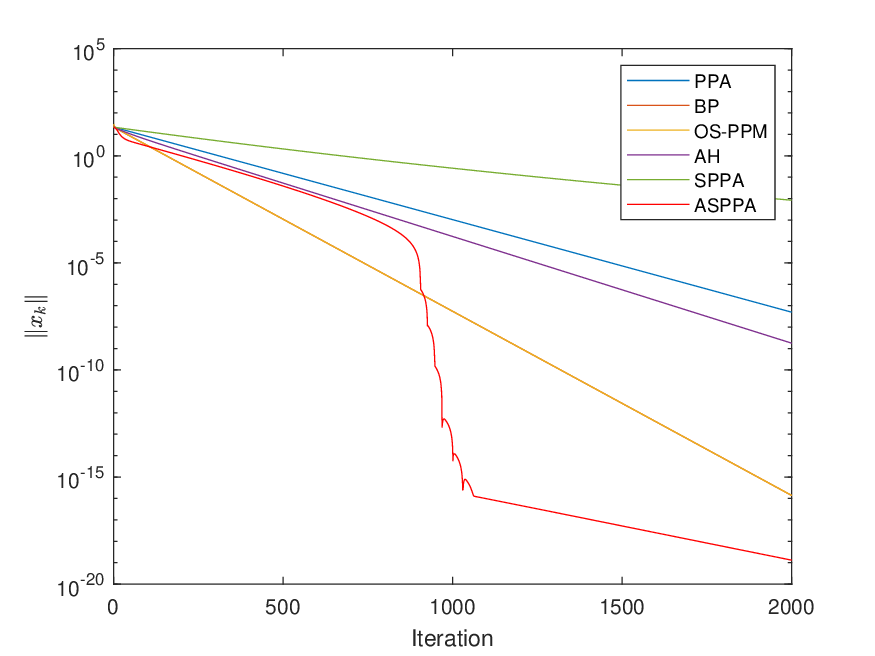}
	\end{subfigure}
	\begin{subfigure}{0.45\linewidth}
		\includegraphics[width=\linewidth]{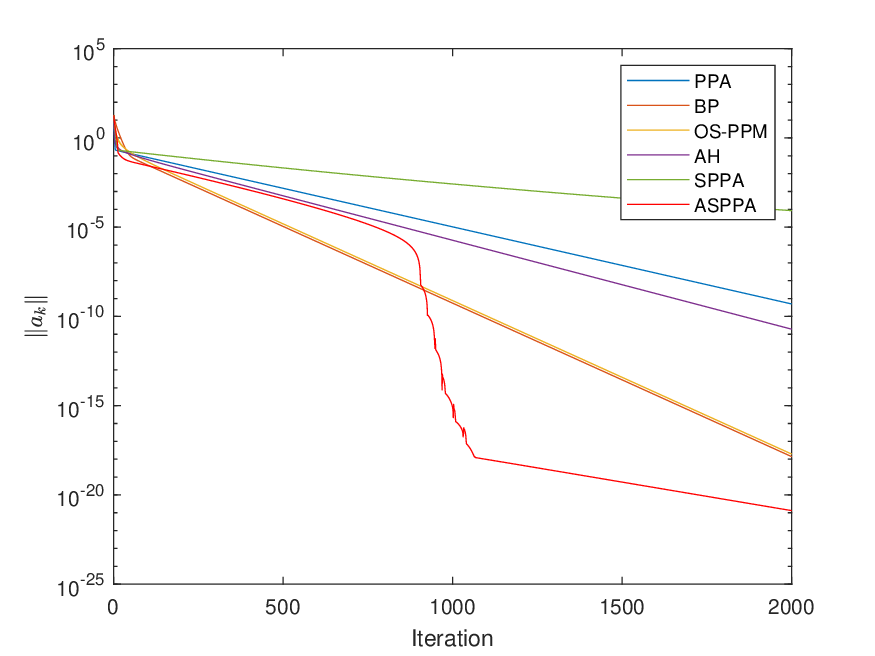}
	\end{subfigure}
	\begin{subfigure}{0.45\linewidth}
		\includegraphics[width=\linewidth]{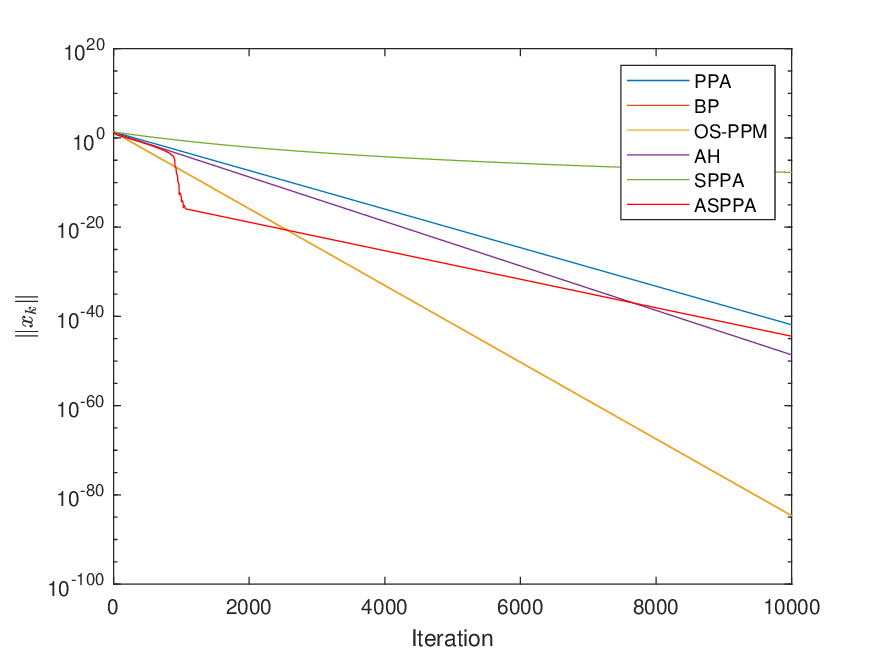}
	\end{subfigure}
	\begin{subfigure}{0.45\linewidth}
		\includegraphics[width=\linewidth]{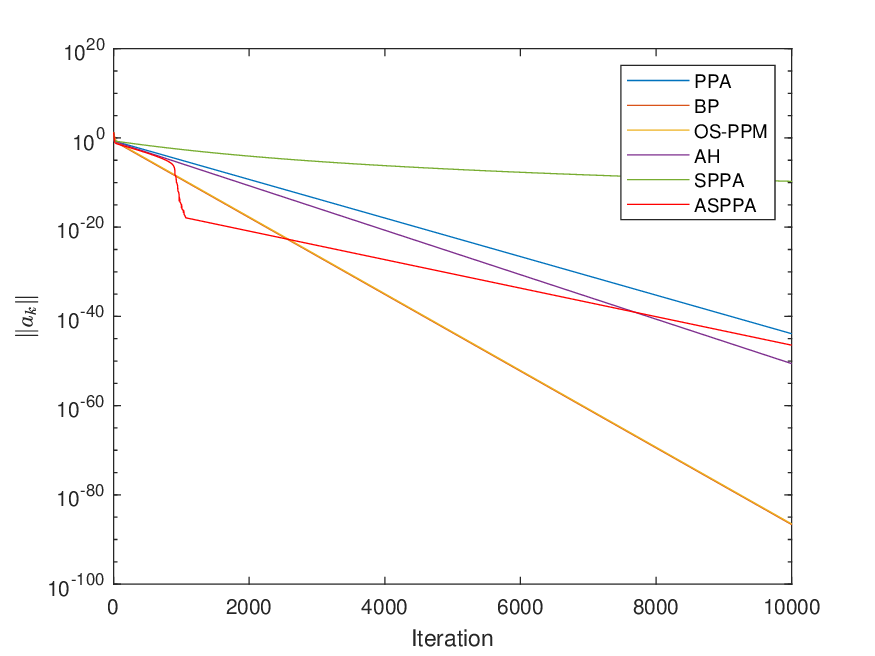}
	\end{subfigure}
	\caption{Comparison between PPA and its five extensions. The two figures in the first row show the convergence curves of $\norm{x_k}$ and $\norm{a_k}$ over the first 2000 iterations. The two figures in the second row show the convergence curves of $\norm{x_k}$ and $\norm{a_k}$ over the entire 10000 iterations. The curves of OS-PPM and BP are very close.}
	\label{fig:1}
\end{figure}

As shown in Figure \ref{fig:1}, ASPPA converges faster in the initial iterations. Moreover, OS-PPM and the Banach--Picard iteration with \eqref{eq:contraction} converge faster when $\norm{x_k}<10^{-20}$. In conclusion, OS-PPM and the Banach--Picard iteration have better asymptotic convergence rates, while ASPPA exhibits faster initial convergence.

\subsection{LASSO Problem}

In this section, we consider the LASSO problem:
\begin{equation}
	\min_u \frac{1}{2}\norm{Mu-b}^2+\lambda\norm{u}_1,
	\label{eq:lasso}
\end{equation}
where $M\in\mathbb{R}^{m\times n}$, $b\in\mathbb{R}^m$, and $\lambda>0$. Problem \eqref{eq:lasso} is equivalent to
\begin{equation}
	\min_{u,v}\frac{1}{2}\norm{Mu-b}^2+\lambda\norm{v}_1,\quad \text{s.t. } u=v.
	\label{eq:lasso2}
\end{equation}
ADMM is an effective method for solving \eqref{eq:lasso2} and is an instance of PPA \cite{eckstein92}. It has been proven that the convergence rate of ADMM for solving \eqref{eq:lasso2} is linear \cite{hong2017}.

In our experiment, the matrix $M$ and vector $b$ are randomly generated, and $\lambda=\lambda_c\norm{M^\top b}_\infty$ with $\lambda_c\in(0,1)$; the choice of $\lambda$ follows \cite{li18}. We test ADMM (with penalty parameter 1) and three accelerated ADMM variants: AH-ADMM derived from the adaptive Halpern iteration, S-ADMM derived from SPPA with $r=20$, $c=(r-1)/2$, and AS-ADMM derived from Algorithm \ref{al:sppa} with $r=20$, $c=(r-1)/2$, $\alpha=0$. We run these four algorithms for 100000 iterations and plot the curves of $\norm{x_k-y_k}$ and $f(x_k)-f^*$, where $f^*$ is the optimal value of \eqref{eq:lasso} computed by a high-accuracy solver, versus the iteration count in Figure \ref{fig:2}. As shown in Figure \ref{fig:2}, AS-ADMM converges fastest.

\begin{figure}[ht]
	\centering
	\begin{subfigure}{0.45\linewidth}
		\includegraphics[width=\linewidth]{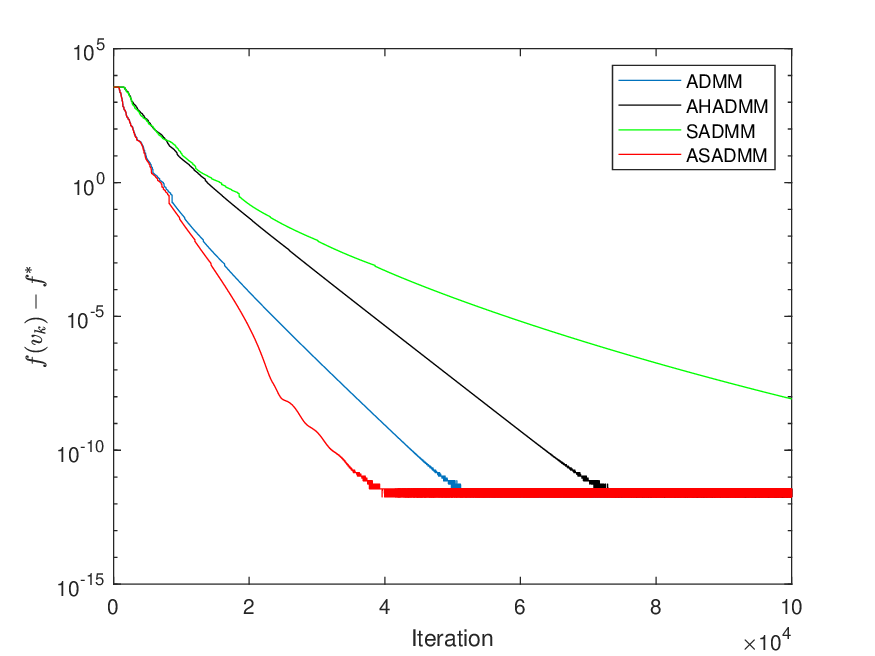}
	\end{subfigure}
	\begin{subfigure}{0.45\linewidth}
		\includegraphics[width=\linewidth]{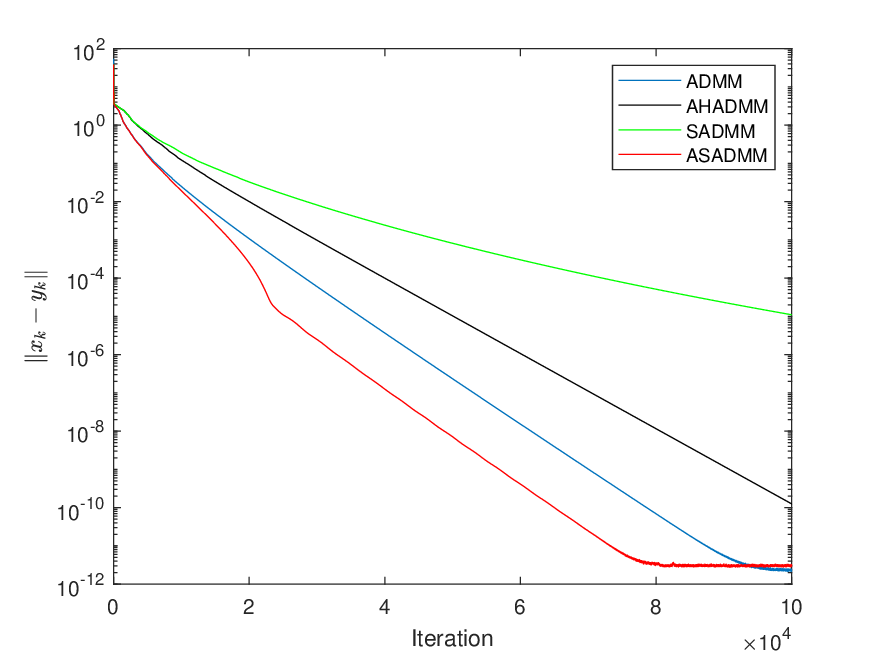}
	\end{subfigure}
	\caption{Comparison between ADMM and its three accelerated variants for solving \eqref{eq:lasso2}.}
	\label{fig:2}
\end{figure}

As shown in Figure \ref{fig:2}, ADMM exhibits a linear convergence rate and converges faster than AH-ADMM and S-ADMM. However, AS-ADMM converges even faster than ADMM.

Next, we test another effective algorithm for solving the LASSO problem. By invoking the Fenchel conjugate function, \eqref{eq:lasso} is equivalent to
\begin{equation}
	\min_{u\in\mathbb{R}^n}\max_{v\in\mathbb{R}^m} \lambda\norm{u}_1+\inner{Mu,v}-\frac{1}{2}\norm{v}^2-\inner{b,v}.
	\label{eq:lasso3}
\end{equation}
PDHG can be used to solve \eqref{eq:lasso3} and is an instance of PPA \cite{chambolle16,connor20}. Similarly, in our experiment, $M$ and $b$ are randomly generated, and $\lambda=\lambda_c\norm{M^\top b}_\infty$ with $\lambda_c\in(0,1)$; the choice of $\lambda$ follows \cite{li18}. We test PDHG and three accelerated PDHG variants: AH-PDHG derived from the adaptive Halpern iteration, S-PDHG derived from SPPA with $r=20$, $c=(r-1)/2$, and AS-PDHG derived from Algorithm \ref{al:sppa} with $r=20$, $c=(r-1)/2$, $\alpha=0$. We run these four algorithms for 100000 iterations and plot the curves of $\norm{x_k-y_k}$ and $f(x_k)-f^*$, where $f^*$ is the optimal value of \eqref{eq:lasso} computed by a high-accuracy solver, versus the iteration count in Figure \ref{fig:3}. As shown in Figure \ref{fig:3}, AS-PDHG converges fastest.

\begin{figure}[ht]
	\centering
	\begin{subfigure}{0.45\linewidth}
		\includegraphics[width=\linewidth]{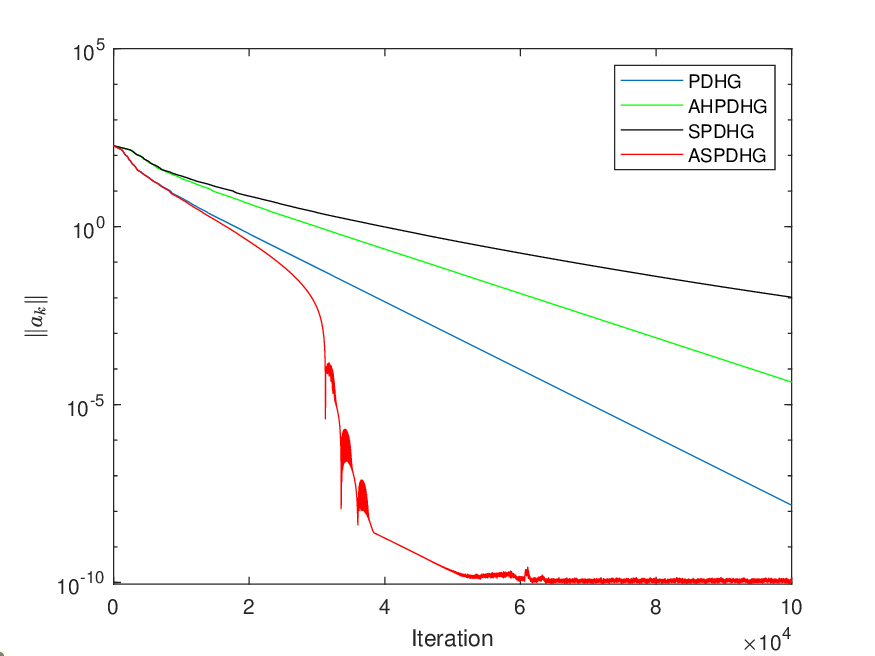}
	\end{subfigure}
	\begin{subfigure}{0.45\linewidth}
		\includegraphics[width=\linewidth]{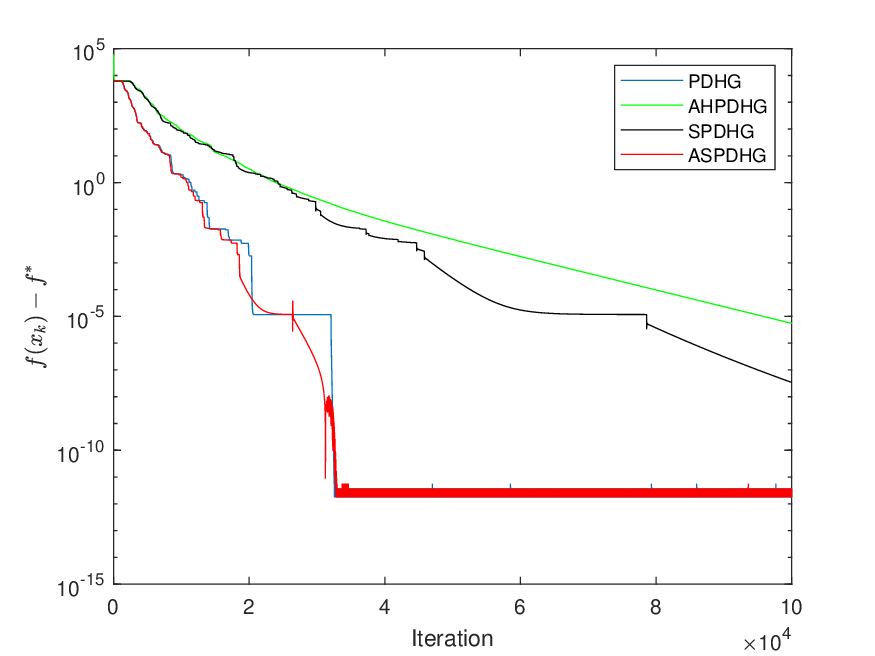}
	\end{subfigure}
	\caption{Comparison between PDHG and its three accelerated variants for solving \eqref{eq:lasso3}.}
	\label{fig:3}
\end{figure}

As shown in Figure \ref{fig:3}, PDHG and AS-PDHG converge faster than AH-PDHG and S-PDHG. However, PDHG and AS-PDHG have similar performance in decreasing the function value. A possible reason is that Algorithm \ref{al:sppa} has better performance in decreasing $\norm{a_k}$ (see the left subfigure in Figure \ref{fig:3}).

\subsection{Basis Pursuit}

In the previous section, we mainly tested Algorithm \ref{al:sppa} with $\alpha=0$ and found that it exhibits better numerical performance. This raises the question of whether $\alpha=0$ yields the fastest convergence rate for Algorithm \ref{al:sppa}. Here, we present an experiment showing that $\alpha=0$ is not always the best choice.

We consider the basis pursuit problem
\begin{equation}
	\min_u \norm{u}_1,\quad \text{s.t. } Mu=b,
\end{equation}
where $M\in\mathbb{R}^{m\times n}$ and $b\in\mathbb{R}^m$. According to the Lagrangian duality theorem, it is equivalent to solving
\begin{equation}
	\min_{u\in\mathbb{R}^n}\max_{v\in\mathbb{R}^m}\norm{u}_1+\inner{v,Mu-b}.
	\label{eq:basispursuit}
\end{equation}

In our experiment, $M$ and $b$ are randomly generated. We test ADMM (with penalty parameter 1) and three accelerated ADMM variants: AH-ADMM derived from the adaptive Halpern iteration, S-ADMM derived from SPPA with $r=20$, $c=(r-1)/2$. We also test two instances of accelerated PDHG derived from Algorithm \ref{al:sppa} with $r=20$, $c=(r-1)/2$, and $\alpha=0$ and $\alpha=0.99$, respectively, labeled as AS-PDHG ($\alpha=0$) and AS-PDHG ($\alpha=0.99$).

\begin{figure}[ht]
	\centering
	\includegraphics[width=.6\linewidth]{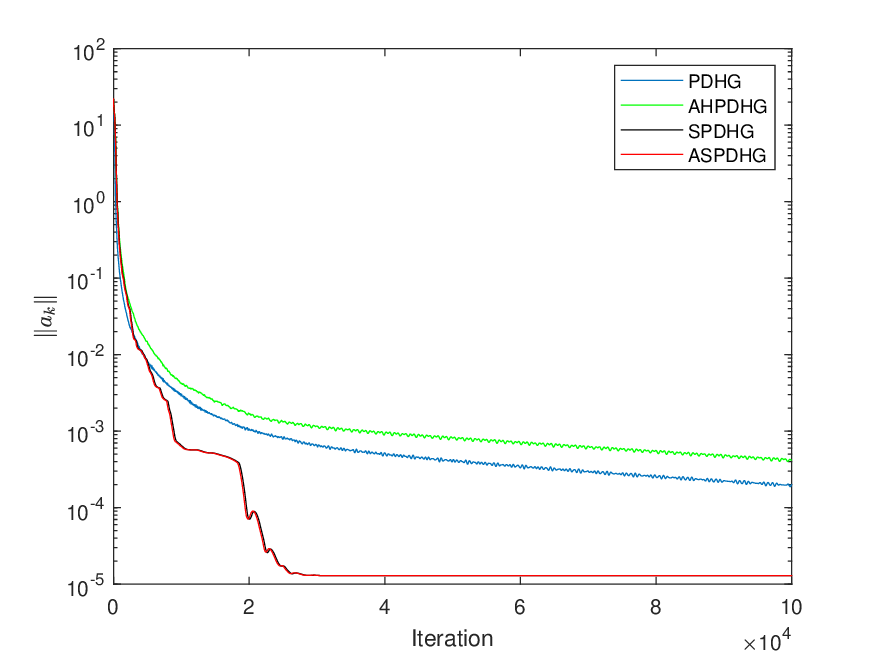}
	\caption{Comparison between PDHG and its three accelerated variants.}
	\label{fig:4}
\end{figure}

As shown in Figure \ref{fig:4}, S-PDHG and AS-PDHG converge faster than AH-PDHG and PDHG. Moreover, S-PDHG and AS-PDHG have similar numerical performance, with AS-PDHG slightly better than S-PDHG. Next, we compare the numerical performance of AS-PDHG with different choices of $\alpha$. The setting is the same as in the previous experiment, with $r=20$ and $c=(r-1)/2$. We run AS-PDHG with $\alpha=0, 0.25, 0.5, 0.75, 0.99$ and plot $\norm{a_k}$ against the iteration count in Figure \ref{fig:5}.

\begin{figure}[ht]
	\centering
	\includegraphics[width=.6\linewidth]{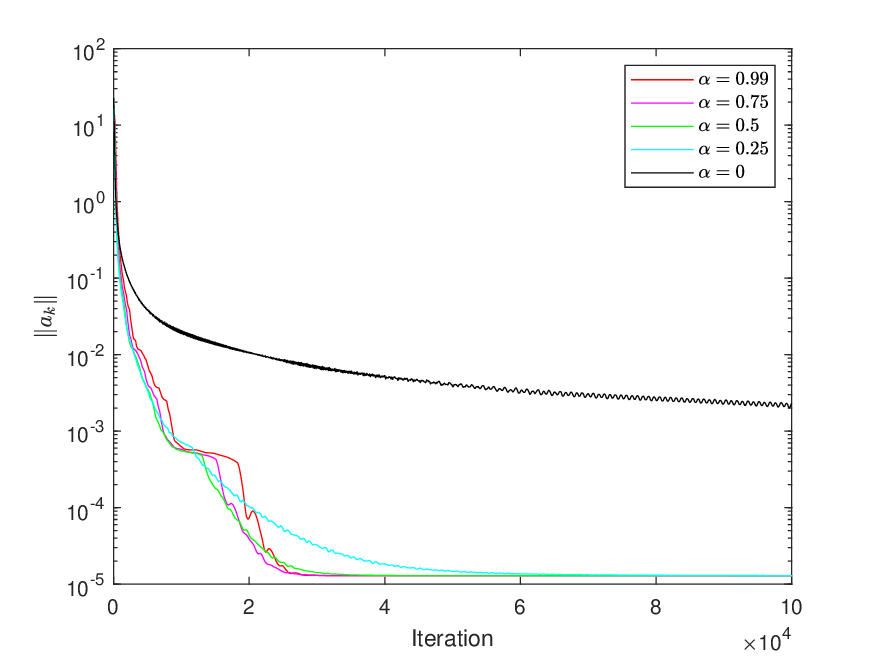}
	\caption{Comparison between AS-PDHG with different $\alpha$ for solving \eqref{eq:basispursuit}.}
	\label{fig:5}
\end{figure}

As shown in Figure \ref{fig:5}, when $\alpha=0$, AS-PDHG converges most slowly. The variants with $\alpha=0.25, 0.5, 0.75, 0.99$ have similar numerical performance. A possible reason is that AS-PDHG itself converges faster than PDHG for solving \eqref{eq:basispursuit}, which implies that the $o(1/k^2)$ convergence rate of AS-PDHG makes it converge faster. Thus, to obtain a numerical convergence rate similar to that of AS-PDHG, $\alpha>0$ is required to guarantee that AS-PDHG satisfies $\liminf_{k\to\infty}k^2\norm{a_k}^2=0$, so that the resulting algorithm converges faster.
\section{Conclusions}
In this paper, we proposed an adaptive symplectic proximal point algorithm (ASPPA) for solving monotone inclusion problems. The algorithm adaptively selects the parameter $\varphi_{k+1}$ by solving a quadratic equation. We proved that the squared norm term satisfies $\norm{a_k}^2=O(1/k^2)$ for maximally monotone operators. When $\alpha\in(0,1)$, we further showed that $\liminf_{k\to\infty}k^2\norm{a_k}^2=0$, and under an additional summability condition, we obtained the stronger $o(1/k^2)$ convergence rate. For strongly monotone operators, we established a linear (exponential) convergence rate without requiring prior knowledge of the strong monotonicity parameter $\mu$. Numerical experiments show that ASPPA exhibits better numerical performance than several existing accelerated methods.
\section*{Acknowledgments}

The author is deeply grateful to Professor Ya-xiang Yuan for his continuous guidance and encouragement during the author's doctoral studies. This work is based on the author's doctoral research supervised by Professor Yuan. The author also sincerely thanks him for his careful reading of the present paper and for his valuable comments and suggestions. The author used DeepSeek, a large language model, for language polishing and editing of this manuscript.

\appendix
\section{Sufficient Conditions for the $o(1/k^2)$ Convergence Rate}
\label{sec:a}

We first show that the condition
\begin{equation}
	\sum_{k=0}^\infty(\varphi_{k+1}-\varphi_k-1)^2\varphi_k\norm{a_{k+1}}^2<+\infty
	\label{eq:sumbeta}
\end{equation}
is sufficient for the convergence rate of Algorithm \ref{al:sppa} to be $o(1/k^2)$ and for the sequence $\{x_k\}$ generated by Algorithm \ref{al:sppa} to converge weakly to a point in $\zero(A)$. We first present some auxiliary results.

\begin{proposition}[Cauchy Condensation Test]
	Given a nonnegative sequence $\{a_k\}$, $\sum_{k=0}^\infty a_k<+\infty$ if and only if $\sum_{k=0}^\infty 2^ka_{2^k}<+\infty$.
\end{proposition}

\begin{proposition}[Lemma 5.31 in \cite{bauschke17}]
	\label{prop:sumprop}
	Let $\{a_k\}_{k\geq 1}$, $\{b_k\}_{k\geq 1}$, and $\{d_k\}_{k\geq 1}$ be sequences of real numbers. Assume that $\{a_k\}_{k\geq 1}$ is bounded from below, and that $\{b_k\}_{k\geq 1}$ and $\{d_k\}_{k\geq 1}$ are nonnegative sequences such that $\sum_{k=1}^\infty d_k<+\infty$. If
	\[
	a_{k+1}\leq a_k-b_k+d_k,\quad\forall k\geq 1,
	\]
	then the following statements hold:
	\begin{itemize}
		\item the sequence $\{b_k\}_{k\geq 1}$ is summable, i.e., $\sum_{k=1}^\infty b_k<+\infty$;
		\item the sequence $\{a_k\}_{k\geq 1}$ is convergent.
	\end{itemize}
\end{proposition}

\begin{proposition}[Opial's Lemma in \cite{opial67}, Lemma 2.47 in \cite{bauschke17}]
	\label{prop:discreteopial}
	Let $\{x_k\}$ be a sequence in $\mathcal{H}$ and let $S$ be a nonempty subset of $\mathcal{H}$. Suppose that, for every $x\in S$, $\norm{x_k-x}$ converges and that every weak sequential cluster point of $\{x_k\}$ belongs to $S$. Then $\{x_k\}$ converges weakly to a point in $S$.
\end{proposition}

\begin{proposition}[Proposition 20.38 (ii) in \cite{bauschke17}]
	\label{prop:nonexpansiveweak}
	Let $A$ be a maximally monotone operator. If there exist two sequences $\{x_k\}$ and $\{u_k\}$ such that $u_k\in A(x_k)$, $x_k\rightharpoonup x$, and $u_k\to u$, then $u\in A(x)$.
\end{proposition}

\begin{lemma}
	\label{lem:zbounded}
	Let $\{z_k\}$ be the sequence generated by Algorithm \ref{al:sppa}. If $A$ is maximally monotone, $r>1$, and $c\in(0,r-1)$, then for any $x^*\in\zero(A)$, the sequence $\{\norm{z_k-x^*}^2\}$ is uniformly bounded.
\end{lemma}

\begin{proof}
	Similar to Lemma \ref{lem:lyapunovdiscretenonnegative}, the Lyapunov function $\E(k)$ defined in \eqref{eq:lyapunovdiscrete} can be rewritten as
	\[
	\E(k)=\frac{\varphi_k}{2r}\norm{a_k}^2+\frac{\varphi_k}{r}\inner{a_k,x_k-x^*}+\frac{r-1-c}{2c}\norm{z_k-x^*}^2+\frac{1}{2}\norm{\frac{\varphi_k}{r}a_k+x^*-z_k}^2.
	\]
	Since $\{\E(k)\}$ is non-increasing (Theorem \ref{thm:lyapunovdiscrete}), we have
	\[
	\frac{r-1-c}{2c}\norm{z_k-x^*}^2\leq\E(k)\leq\E(0)=\frac{r-1}{2c}\norm{x_0-x^*}^2.
	\]
	Hence the sequence $\{\norm{z_k-x^*}^2\}$ is bounded by $\frac{r-1}{r-1-c}\norm{x_0-x^*}^2$.
\end{proof}

\begin{theorem}
	\label{thm:pointconvergence}
	Let $\{x_k\}$ and $\{a_k\}$ be the sequences generated by Algorithm \ref{al:sppa}. If $A$ is maximally monotone, $r>1$, $c\in(0,r-1)$, $\alpha\in(0,1)$, and \eqref{eq:sumbeta} holds, then $\{x_k\}$ converges weakly to a point in $\zero(A)$ and
	\[
	\lim_{k\to\infty}k^2\norm{a_k}^2=0.
	\]
\end{theorem}

\begin{proof}
	Step 1: Consider the auxiliary sequence
	\[
	\G(k)=\frac{1}{2}\norm{x_k-z_k}^2+\frac{(r-c)(\varphi_k+r)^2}{2c\varphi_k^2}\norm{z_k-x^*}^2,\quad\forall k\geq 1,
	\]
	and prove that
	\begin{equation}
		\sum_{k=1}^{\infty}\frac{\norm{x_{k+1}-z_{k+1}}^2}{\varphi_k}<+\infty,
		\label{eq:xzsum}
	\end{equation}
	and that \begin{equation}
		\lim_{k\to\infty}\norm{x_k-z_k}^2 \text{ exists.}
		\label{eq:xzlimit}
	\end{equation}
	
	First, using
	\begin{align}
		x_k &= \frac{\varphi_k+r}{\varphi_k}x_{k+1}-\frac{r}{\varphi_k}z_k+\frac{\varphi_k+r}{\varphi_k}a_{k+1},\quad\forall k\geq 1, \label{eq:xdiff}\\
		z_{k+1} &= z_k-\frac{(\varphi_{k+1}-\varphi_k)c}{r}a_{k+1},\label{eq:zdiff}
	\end{align}
	we obtain
	\begin{align*}
		&\ \frac{1}{2}\norm{x_{k+1}-z_{k+1}}^2-\frac{1}{2}\norm{x_k-z_k}^2 \\
		=&\ \frac{1}{2}\norm{x_{k+1}-z_{k+1}}^2-\frac{1}{2}\norm{\frac{\varphi_k+r}{\varphi_k}(x_{k+1}-z_{k+1})+\left(\frac{\varphi_k+r}{\varphi_k}-\frac{(\varphi_{k+1}-\varphi_k)(\varphi_k+r)c}{r\varphi_k}\right)a_{k+1}}^2\\
		=&\ -\frac{r(2\varphi_k+r)}{2\varphi_k}\norm{x_{k+1}-z_{k+1}}^2-\frac{1}{2}\left(\frac{\varphi_k+r}{\varphi_k}-\frac{(\varphi_{k+1}-\varphi_k)(\varphi_k+r)c}{r\varphi_k}\right)^2\norm{a_{k+1}}^2\\
		&\ -\frac{\varphi_k+r}{\varphi_k}\left(\frac{\varphi_k+r}{\varphi_k}-\frac{(\varphi_{k+1}-\varphi_k)(\varphi_k+r)c}{r\varphi_k}\right)\inner{a_{k+1},x_{k+1}-z_{k+1}}.
	\end{align*}
	By Young's inequality,
	\begin{align*}
		&\ \frac{(\varphi_{k+1}-\varphi_k-1)(\varphi_k+r)^2}{\varphi_k^2}\inner{a_{k+1},x_{k+1}-z_{k+1}}\\
		\leq&\ \frac{(\varphi_{k+1}-\varphi_k-1)^2(\varphi_k+r)^4}{\varphi_k^2(2\varphi_k+r)}\norm{a_{k+1}}^2+\frac{r(2\varphi_k+r)}{4\varphi_k^2}\norm{x_{k+1}-z_{k+1}}^2.
	\end{align*}
	Therefore,
	\begin{align*}
		&\ \frac{1}{2}\norm{x_{k+1}-z_{k+1}}^2-\frac{1}{2}\norm{x_k-z_k}^2 \\
		\leq &\ -\frac{r(2\varphi_k+r)}{4\varphi_k^2}\norm{x_{k+1}-z_{k+1}}^2-\frac{(r-c)(\varphi_{k+1}-\varphi_k)(\varphi_k+r)^2}{r\varphi_k^2}\inner{a_{k+1},x_{k+1}-z_{k+1}}\\
		&\ +\left[\frac{(\varphi_{k+1}-\varphi_k-1)^2(\varphi_k+r)^4}{\varphi_k^2(2\varphi_k+r)}-\frac{(\varphi_k+r)^2}{2\varphi_k^2}\left(1-\frac{(\varphi_{k+1}-\varphi_k)c}{r}\right)^2\right]\norm{a_{k+1}}^2.
	\end{align*}
	
	Next, since $f(y):=\frac{(y+r)^2}{y^2}$ is non-increasing on $(0,+\infty)$, we have
	\begin{align*}
		&\ \frac{(r-c)(\varphi_{k+1}+r)^2}{2c\varphi_{k+1}^2}\norm{z_{k+1}-x^*}^2-\frac{(r-c)(\varphi_k+r)^2}{2c\varphi_k^2}\norm{z_k-x^*}^2\\
		\leq&\ \frac{(r-c)(\varphi_k+r)^2}{2c\varphi_k^2}\norm{z_{k+1}-x^*}^2-\frac{(r-c)(\varphi_k+r)^2}{2c\varphi_k^2}\norm{z_k-x^*}^2\\
		\leq&\ -\frac{(r-c)(\varphi_{k+1}-\varphi_k)(\varphi_k+r)^2}{r\varphi_k^2}\inner{a_{k+1},z_{k+1}-x^*}\\
		&\ -\frac{(r-c)c(\varphi_{k+1}-\varphi_k)^2(\varphi_k+r)^2}{2r^2\varphi_k^2}\norm{a_{k+1}}^2.
	\end{align*}
	Combining the above, we obtain
	\begin{align*}
		&\ \G(k+1)-\G(k)\\
		\leq&\ -\frac{r(2\varphi_k+r)}{4\varphi_k^2}\norm{x_{k+1}-z_{k+1}}^2-\frac{(r-c)(\varphi_{k+1}-\varphi_k)(\varphi_k+r)^2}{r\varphi_k^2}\inner{a_{k+1},x_{k+1}-x^*}\\
		&\ +\left[\frac{(\varphi_{k+1}-\varphi_k-1)^2(\varphi_k+r)^4}{\varphi_k^2(2\varphi_k+r)}-\frac{(\varphi_k+r)^2}{2\varphi_k^2}\left(1-\frac{(\varphi_{k+1}-\varphi_k)c}{r}\right)^2\right]\norm{a_{k+1}}^2\\
		&\ -\frac{(r-c)c(\varphi_{k+1}-\varphi_k)^2(\varphi_k+r)^2}{2r^2\varphi_k^2}\norm{a_{k+1}}^2.
	\end{align*}
	By \eqref{eq:sumbeta},
	\[
	\sum_{k=1}^\infty\frac{(\varphi_{k+1}-\varphi_k-1)^2(\varphi_k+r)^4}{\varphi_k^2(2\varphi_k+r)}\norm{a_{k+1}}^2<+\infty.
	\]
	Applying Proposition \ref{prop:sumprop}, we obtain
	\[
	\sum_{k=1}^{\infty}\frac{\norm{x_{k+1}-z_{k+1}}^2}{\varphi_k}<+\infty,\qquad\lim_{k\to\infty}\G(k)\text{ exists.}
	\]
	
	Step 2: We now deduce the $o(1/k^2)$ convergence rate and the weak convergence. We consider two cases.
	
	Case 1: $\sum_{k=1}^\infty\frac{1}{\varphi_k}<+\infty$. We first prove that $\lim_{k\to\infty}\frac{k}{\varphi_k}=0$. Since $\varphi_{k+1}\geq\varphi_k+1$, by the Cauchy condensation test we have
	\[
	\lim_{n\to\infty}\frac{2^n}{\varphi_{2^n}}=0.
	\]
	For every positive integer $k$, there exists $n$ such that $2^n\leq k<2^{n+1}$. Hence
	\[
	0\leq\frac{k}{\varphi_k}\leq\frac{k}{\varphi_{2^n}}\leq\frac{2^{n+1}}{\varphi_{2^n}}=2\cdot\frac{2^n}{\varphi_{2^n}}\to 0.
	\]
	Therefore $\lim_{k\to\infty}\frac{k}{\varphi_k}=0$. Since $\norm{a_k}^2\leq O(1/\varphi_k^2)$, we obtain $\lim_{k\to\infty}k^2\norm{a_k}^2=0$.
	
	Next, we discuss the weak convergence. For any $x^*\in\zero(A)$, the standard three-point identity gives
	\begin{equation}
		\frac{1}{2}\norm{x_{k+1}-x^*}^2
		=\frac{1}{2}\norm{x_k-x^*}^2
		+\inner{x_{k+1}-x_k,x_{k+1}-x^*}
		-\frac{1}{2}\norm{x_{k+1}-x_k}^2.
		\label{eq:threepoint}
	\end{equation}
	Substituting \eqref{eq:xdiff} and \eqref{eq:zdiff} into the inner product term in \eqref{eq:threepoint}, we obtain
	\begin{align} \inner{x_{k+1}-x_k,x_{k+1}-x^*}=\frac{r}{\varphi_k+r}\inner{z_k-x_k,x_{k+1}-x^*}
		-\inner{a_{k+1},x_{k+1}-x^*}.
		\label{eq:inner1}
	\end{align}
	Similarly, substituting \eqref{eq:xdiff} into the squared norm term in \eqref{eq:threepoint} gives
	\begin{align}
		&\ -\frac{1}{2}\norm{x_{k+1}-x_k}^2\notag\\
		=&\ -\frac{r^2}{2\varphi_k^2}\norm{z_k-x_{k+1}}^2
		+\frac{r(\varphi_k+r)}{\varphi_k^2}\inner{z_k-x_{k+1},a_{k+1}}-\frac{(\varphi_k+r)^2}{2\varphi_k^2}\norm{a_{k+1}}^2.
		\label{eq:norm1}
	\end{align}
	
	Combining \eqref{eq:threepoint}, \eqref{eq:inner1}, and \eqref{eq:norm1}, we have
	\begin{align}
		&\ \frac{1}{2}\norm{x_{k+1}-x^*}^2\notag\\
		=&\ \frac{1}{2}\norm{x_k-x^*}^2
		+\frac{r}{\varphi_k+r}\inner{z_k-x_k,x_{k+1}-x^*}-\inner{a_{k+1},x_{k+1}-x^*}\notag\\ 
		&\ -\frac{r^2}{2\varphi_k^2}\norm{z_k-x_{k+1}}^2+\frac{r(\varphi_k+r)}{\varphi_k^2}\inner{z_k-x_{k+1},a_{k+1}}-\frac{(\varphi_k+r)^2}{2\varphi_k^2}\norm{a_{k+1}}^2\notag\\
		=&\ \frac{1}{2}\norm{x_k-x^*}^2
		+\frac{r}{\varphi_k+r}\inner{z_k-x_k,x_{k+1}-x^*}-\left(1+\frac{r(\varphi_k+r)}{\varphi_k^2}\right)\inner{a_{k+1},x_{k+1}-x^*}\notag\\ 
		&\ -\frac{r^2}{2\varphi_k^2}\norm{z_k-x_{k+1}}^2+\frac{r(\varphi_k+r)}{\varphi_k^2}\inner{z_k-x^*,a_{k+1}}-\frac{(\varphi_k+r)^2}{2\varphi_k^2}\norm{a_{k+1}}^2.
		\label{eq:expand}
	\end{align}
	Since $A$ is maximally monotone, $a_{k+1}\in A(x_{k+1})$, and $0\in A(x^*)$, we have
	\[
	\inner{a_{k+1},x_{k+1}-x^*}\geq 0.
	\]
	Consequently,
	\begin{equation*}
		-\left(1+\frac{r(\varphi_k+r)}{\varphi_k^2}\right)\inner{a_{k+1},x_{k+1}-x^*}\leq 0.
	\end{equation*}
	Also, we have that \begin{align*}
		 &\ \frac{r}{\varphi_k+r}\inner{z_k-x_k,x_{k+1}-x^*} \\
		\leq&\ \frac{r}{\varphi_k+r}|\inner{z_k-x_k,x_{k+1}-x^*}|\\
		= &\ \frac{r}{\varphi_k+r}|\inner{z_k-x_k,x_{k+1}-z_{k+1}}+\frac{r}{\varphi_k+r}\inner{z_k-x_k,z_{k+1}-x^*}|\\
		\leq&\ \frac{r}{\varphi_k+r}|\inner{z_k-x_k,x_{k+1}-z_{k+1}}|+\frac{r}{\varphi_k+r}|\inner{z_k-x_k,z_{k+1}-x^*}|\\
		\leq&\ \frac{r}{2(\varphi_k+r)}\left(\norm{x_k-z_k}^2+\norm{x_{k+1}-z_{k+1}}^2\right)+\frac{r}{\varphi_k+r}|\inner{z_k-x_k,z_{k+1}-x^*}|.
	\end{align*}
	Since $\{\norm{z_k-x^*}^2\}$ is bounded, by H\"older's and Cauchy's inequalities, we have
	\begin{align*}
		&\ \sum_{k=1}^\infty\frac{r}{\varphi_k+r}|\inner{z_k-x_k,z_{k+1}-x^*}|\\
		\leq&\ r\sup_{k\geq 1}\norm{z_k-x^*}\left(\sum_{k=1}^\infty\frac{1}{\varphi_k+r}\right)^{1/2}\left(\sum_{k=1}^\infty\frac{\norm{z_k-x_k}^2}{\varphi_k+r}\right)^{1/2}<+\infty
	\end{align*}
	Similarly, by H\"older's and Cauchy's inequalities, we have\[
	\sum_{k=1}^{+\infty}\frac{r(\varphi_k+r)}{\varphi_k^2}\inner{z_k-x^*,a_{k+1}}\leq\sum_{k=1}^{+\infty}\frac{r(\varphi_k+r)}{\varphi_k^2}|\inner{z_k-x^*,a_{k+1}}|<+\infty
	\]
	By Proposition \ref{prop:sumprop}, $\lim_{k\to\infty}\frac{1}{2}\norm{x_k-x^*}^2$ exists for all $x^*\in\zero(A)$.
	
	For any weak cluster point $x_\infty$ of $\{x_k\}$, let $\{x_{k_j}\}$ be a subsequence such that $x_{k_j}\rightharpoonup x_\infty$. Since $a_k\to 0$, Proposition \ref{prop:nonexpansiveweak} implies $x_\infty\in\zero(A)$. By Opial's lemma, $x_k\rightharpoonup x_\infty\in\zero(A)$.
	
	Case 2: $\sum_{k=1}^\infty\frac{1}{\varphi_k}=+\infty$ and $\alpha>0$. We first prove that $\lim_{k\to\infty}\norm{z_k-x^*}^2$ exists for all $x^*\in\zero(A)$. Using $z_{k+1}=z_k-\frac{c(\varphi_{k+1}-\varphi_k)}{r}a_{k+1}$, we have
	\begin{align*}
		&\ \frac{1}{2}\norm{z_{k+1}-x^*}^2\\
		=&\ \frac{1}{2}\norm{z_k-x^*}^2-\frac{c(\varphi_{k+1}-\varphi_k)}{r}\inner{a_{k+1},x_{k+1}-x^*}\\
		&\ +\frac{c(\varphi_{k+1}-\varphi_k)}{r}\inner{a_{k+1},x_{k+1}-z_{k+1}}-\frac{1}{2}\norm{z_{k+1}-z_k}^2.
	\end{align*}
	Since $\varphi_{k+1}-\varphi_k\geq 1$ and $x^*\in\zero(A)$, we have
	\[
	-\frac{c(\varphi_{k+1}-\varphi_k)}{r}\inner{a_{k+1},x_{k+1}-x^*}\leq 0.
	\]
	By Young's inequality,
	\begin{align*}
		&\ \frac{c(\varphi_{k+1}-\varphi_k)}{r}\inner{a_{k+1},x_{k+1}-z_{k+1}}\\
		\leq&\ \frac{c^2(\varphi_{k+1}-\varphi_k)^2\varphi_k}{r^2}\norm{a_{k+1}}^2+\frac{\norm{x_{k+1}-z_{k+1}}^2}{\varphi_k}\\
		=&\ \frac{c^2\varphi_k[(\varphi_{k+1}-\varphi_k-1)^2+2(\varphi_{k+1}-\varphi_k-1)\varphi_k+\varphi_k]}{r^2}\norm{a_{k+1}}^2+\frac{\norm{x_{k+1}-z_{k+1}}^2}{\varphi_k}.
	\end{align*}
	Using \eqref{eq:sumalpha}, \eqref{eq:sumbeta} and \eqref{eq:xzsum}, we obtain
	\[
	\sum_{k=1}^\infty\frac{c(\varphi_{k+1}-\varphi_k)}{r}\inner{a_{k+1},x_{k+1}-z_{k+1}}\leq\sum_{k=1}^\infty\frac{c(\varphi_{k+1}-\varphi_k)}{r}|\inner{a_{k+1},x_{k+1}-z_{k+1}}|<+\infty.
	\]
	By Proposition \ref{prop:sumprop}, $\lim_{k\to\infty}\norm{z_k-x^*}^2$ exists.
	
	For any weak cluster point $z_\infty$ of $\{z_k\}$, let $\{z_{k_j}\}$ be a subsequence such that $z_{k_j}\rightharpoonup z_\infty$. Since $\lim_{k\to\infty}\norm{x_k-z_k}^2=0$, we have $x_{k_j}\rightharpoonup z_\infty$. Since $a_k\to 0$, Proposition \ref{prop:nonexpansiveweak} implies $z_\infty\in\zero(A)$. By Opial's lemma, $z_k\rightharpoonup z_\infty\in\zero(A)$ and $x_k\rightharpoonup z_\infty\in\zero(A)$.
	
	Next, we show that $\lim_{k\to\infty}k^2\norm{a_k}^2=0$. Since $\{\E(k)\}$ is non-increasing and nonnegative, $\lim_{k\to\infty}\E(k)$ exists. Due to \eqref{eq:xzsum}, \eqref{eq:xzlimit} and the assumption that is $\sum_{k=1}^{+\infty}1/\varphi_k<+\infty$, we can prove $\lim_{k\to\infty}\norm{x_k-z_k}^2=0$. Since $\lim_{k\to\infty}\norm{x_k-z_k}^2=0$ and $\lim_{k\to\infty}\norm{z_k-x^*}^2$ exists, we have that $\lim_{k\to\infty}\varphi_k^2\norm{a_k}^2$ exists. Moreover, since
	\[
	\sum_{k=0}^\infty(\varphi_{k+1}-\varphi_k)\varphi_k\norm{a_k}^2<+\infty,
	\]
	the limit must be zero. Therefore,
	\[
	0\leq\lim_{k\to\infty}k^2\norm{a_k}^2\leq\lim_{k\to\infty}\varphi_k^2\norm{a_k}^2=0.
	\]
\end{proof}

\begin{remark}
	If $\varphi_k=k$, then $\varphi_{k+1}-\varphi_k-1=0$, so \eqref{eq:sumbeta} is automatically satisfied. Thus the weak convergence of the original SPPA follows.
\end{remark}

\bibliographystyle{amsplain}
\bibliography{reference}

\end{document}